\documentclass[12pt,reqno]{amsart}
\usepackage[english]{babel}

\usepackage{mathrsfs}
\usepackage{amsthm}
\usepackage{amssymb}
\usepackage{amsmath,amsthm,amsfonts}
\usepackage{mathabx}
\usepackage{hhline}
\usepackage{textcomp}
\usepackage{amsmath,amscd}
\usepackage[all,cmtip]{xy}
\usepackage{mathtools}
\usepackage[margin=1in]{geometry}
\usepackage{lipsum}
\usepackage{stmaryrd}
\usepackage{extarrows}

\author[]{Kaveh Eftekharinasab}
\title[]{On the generalization of the Darboux theorem}
\address{Topology lab. \\ Institute of Mathematics of NAS of Ukraine \\ Te\-re\-shchen\-kivska st. 3, Kyiv, 01601 Ukraine}
\email{kaveh@imath.kiev.ua}
\makeatletter
\@addtoreset{equation}{section}
\@addtoreset{figure}{section}
\@addtoreset{table}{section}
\makeatother

\newtheorem{theorem}{Theorem}[section]
\newtheorem{lemma}{Lemma}[section]
\newtheorem{remk}{Remark}[section]
\newtheorem{prop}{Proposition}[section]
\newtheorem{defn}{Definition}[section]
\newtheorem{cor}{Corollary}[section]
\theoremstyle{definition}

\DeclareMathAlphabet{\mathpzc}{OT1}{pzc}{m}{it}

\newcommand{\mc}{MC^k}
\newcommand{\ff}{\mathbb{F}}
\newcommand{\rr}{\mathbb{R}}
\newcommand{\nn}{\mathbb{N}}

\newcommand{\sn}{\parallel}

\DeclareMathOperator{\dd}{d}
\DeclareMathOperator{\id}{id}

\newcommand{\me}{\mathcal{M}}

\begin{document}

\begin{abstract}
We provide sufficient conditions for the existence of  Darboux charts on weakly symplectic bounded Fr\'{e}chet manifolds by using the Moser's trick. 
\end{abstract}
\maketitle
\section{Introduction}
The Darboux theorem has been extended to weakly symplectic  Banach manifolds by using Moser's method, see~\cite{Bam}.
The essence of this method is to obtain an appropriate isotopoy  generated by a time dependent vector field that provides the chart transforming of symplectic forms to constant ones. In order to apply this method to more general context of Fr\'{e}chet manifolds
we need to establish the existence of the flow of a vector filed which in general does not exist.
One successful approach to the differential geometry in Fr\'{e}chet context is
in terms of projective limits of Banach manifolds (see \cite{dd}). In this framework,  
a version of the Darboux theorem is proved in \cite{K3}. 

Another approach to Fr\'{e}chet geometry is to use the stronger notion of differentiability (see~\cite{k}). This differentiability leads to a new category of generalized manifolds, the so call bounded (or $ \mc $) Fr\'{e}chet manifolds. In this paper
we prove that in that context the flow of a vector field exists (Theorem \ref{flow}) and we will apply the Moser's method to obtain the Darboux theorem (Theorem~\ref{dar}).

The obtained theorem might be useful to study the topology of the space  of  Riemannian metrics $ \me $ as it
has the structure of a nuclear bounded Fr\'{e}chet manifold. Theorem [\S48.9,\cite{KM}] asserts that
if $(M,\sigma)$ is a smooth weakly symplectic convenient manifold which admits smooth partitions of unity in $C_\sigma^\infty(M,\rr)$, and which admits `Darboux chart', then the symplectic cohomology equals to the De Rham cohomology: $H^k_\sigma(M)= H^k_{DR}(M)$. 

The manifold $ \me $ admits  smooth partition of unity in $C_\sigma^\infty(M,\rr)$ (it follows from Theorem~\cite[Theorem 16.10]{KM} and Definition~\cite[Definition 16.1]{KM}) so it is interesting to ask if it has a Darboux
chart. This, in turn, rises the question: how to construct on $\me$ weak symplectic forms. It is known that (see~\cite{kal}) expect Hilbert manifolds an infinite dimensional manifold may not admit a Lagrangian splitting so in general the Weinstein's construction (\cite{Wei}) is not applicable. Moreover, the Marsden's idea to construct
a symplectic form on a manifold by using the canonical form on its cotangent bundle also is not applicable as there
is no natural smooth vector bundle structure on the cotangent bundle~\cite[Remark I.3.9]{neeb}. It is not clear yet how to construct
symplectic forms on $ \me $ but it seems that it might arise from a weak Riemannian metric and complex structure, however, that would require
some assumptions and ingredients different from ones in Theorem~\ref{dar}.

\section{Bounded differentiability}
In this section we prove the existence of the local flow of a $ \mc $-vector field, we refer to~\cite{k} for more details on bounded Fr\'{e}chet geometry. We denote by $(F,\rho)$ a Fr\'{e}chet space  whose topology is defined by a complete translational-invariant metric $\rho$. We consider only metrics with absolutely convex balls. 
Note that every Fr\'{e}chet space  admits such a metric, cf~\cite{k}. One reason to choose this
particular metric is that  a metric with this property can give us a collection of seminorms that defines the same topology. More precisely: 
\begin{theorem}[\cite{m}, Theorem 3.4] \label{semi}
	Assume that $(F,\rho)$ is a Fr\'{e}chet space and $\rho$ is a metric with absolutely convex balls. Let $$B_{\frac{1}{i}}^\rho(0) \coloneq \{ y\in F  \mid \rho(y,0) < \tfrac{1}{i}\},$$ and suppose $U_i$'s, $i \in \nn$, are convex 
	subsets of $B_{\frac{1}{i}}^\rho(0)$. Define the Minkowski functionals
	$$
	\parallel v \parallel^i \coloneq \inf \{ \epsilon > 0 \mid \epsilon \in \mathbb{R},\,\dfrac{1}{\epsilon} \cdot v \in U_i\}. 
	$$
	These Minkowski functionals are continuous seminorms on $F$. A collection $\{\parallel v \parallel^i\}_{i \in \nn}$ of these seminorms gives the topology of $F$. 
\end{theorem}
In the sequel we assume that a Fr\'{e}chet space $ F $ is graded with the collection of seminorms $ \sn v \sn_F^n = \sum_{k=1}^{k=n} \sn v \sn^k  $ that defines its topology.

Let $(E,g)$  be another Fr\'{e}chet space. Let $\mathcal{L}_{g,\rho}(E,F)$ be the set of all 
linear maps $ L: E \rightarrow F $ such that 
\begin{equation*}
\mathpzc{Lip} (L )_{g,\rho}\, \coloneq \displaystyle \sup_{x \in E\setminus\{0\}} \dfrac{\rho (L(x),0)}{g( x,0)} < \infty.
\end{equation*}
The transversal-invariant metric 
\begin{equation} \label{metric}  
D_{g,\rho}: \mathcal{L}_{g,\rho}(E,F) \times \mathcal{L}_{g,\rho}(E,F) \longrightarrow [0,\infty) , \,\,
(L,H) \mapsto \mathpzc{Lip}(L -H)_{g,\rho} \,,
\end{equation}
on $\mathcal{L}_{\rho,g}(E,F)$ turns it  into an Abelian topological group.
Let $ U $ an open subset of $ E $, and $ P:U \rightarrow F $
a continuous map. If $P$ is Keller-differentiable, $ \operatorname{d}P(p) \in \mathcal{L}_{\rho,g}(E,F) $ for all $ p \in U $, and the induced map 
$ \operatorname{d}P(p) : U \rightarrow \mathcal{L}_{\rho,g}(E,F)   $ is continuous, then $ P $ is called bounded differentiable. We say $ P $ is $ MC^{0} $ 
and write $ P^0 = P $ if it is continuous. 
We say $P$ is an $ MC^{1} $ and write  $P^{(1)} = P' $ if it is bounded differentiable. We define  for $(k>1) $  maps of class $ MC^k$ recursively, see~\cite{k}. If $ \varphi (t) $ is a continuous path in a Fr\'{e}chet space
we denote its derivative by $ \dfrac{d}{d t} \varphi(t)$.  

Within this framework we  define $\mc$ (bounded) Fr\'{e}chet manifolds, $\mc$-maps of manifolds
and tangent bundles and their $ \mc $-vector fields. 
A $ \mc $-vector field $X$ on  a $ \mc $-Fr\'{e}chet manifold $M$ is a $ \mc $-section 
of the tangent bundle $\pi_{TM} : TM \to M$, i.e. a $\mc$ map $X : M \to TM$ 
with $\pi_{TM} \circ X = \id_M$. We write $\mathcal{V}(M)$ for the space of all vector fields
on $M$. If $f \in MC^\infty(M,E)$ is a smooth function on $M$ with values 
in a Fr\'{e}chet space $E$ and 
$X \in \mathcal{V}(M)$, then we obtain a smooth function on $M$ via 
$$ X.f := \dd f \circ X : M  \to E. $$
For $X, Y \in \mathcal{V}(M)$, there exists a unique a vector
field $ [X,Y] \in \mathcal{V}(M)$ determined by the
property that on each open subset $U \subset M$ we have 
$$ [X,Y].f = X.(Y.f) - Y.(X.f)  $$
for all $f \in MC^\infty(U,\rr)$, see~\cite[Lemma II.3.1]{neeb2}.  

A vector field on an infinite dimensional Fr\'{e}chet manifold may have no, one ore multiple  integral curves.
However, a $ \mc $-vector field always has a unique integral curve.
\begin{prop} \cite[Proposition 5.1]{k}\label{ht}
	Let $U \subseteq F$ be open and let $X : U \rightarrow F$ be a $\mc$-vector field, $k\geq 1$. Then for $p_0 \in U$, there is an integral curve
	$\ell : I \rightarrow F$ at $p_0$. Furthermore, any two such curves are equal on the intersection of their domains. 
\end{prop}
\begin{cor}\cite[Corollary 5.1]{k}\label{s}
	Let $U \subseteq F$ be open and let $X : U \rightarrow F$ be a $MC^k$-vector field, $k\geq 1$. Let $\ff_t(p_0)$ be the solution of $\ell'(t) = X (\ell (t)), \, \ell (t_0) = p_0$.
	Then there is an open neighborhood $U_0$ of $p_0$ and a positive real number $\alpha$  such that for every $q \in U_0$ there exists a unique integral
	curve $\ell(t) = \ff_t(q)$ satisfying  $\ell(0) = q$ and $\ell'(t) = X (\ell (t))$ for all $t \in (-\alpha,\alpha)$.
\end{cor}

\begin{theorem}\label{flow}
	Let $X$ be a $ \mc $-vector field on $ U \subset F $, $ k\geqq 1 $. There exists  a real number $ \alpha > 0 $
	such that for each $ x \in U $ there exists a unique integral curve $\ell_x(t)$ satisfying $ \ell_x(0)=x $  for all $ t \in I = (-\alpha, \alpha) $.
	Furthermore, the mapping
	$\ff : I \times U \rightarrow F$ given by $\ff_t(x)= \ff(t,x)= \ell_x(t) $ is of class $ \mc $.
\end{theorem}
\begin{proof}
	The first part of the proof follows from Corollary~\ref{s}. We now proof the second part.
	Let $ x,y \in U $ be arbitrary and
	define the maps $ \varphi_n (t) = \sn \ff (t,x) - \ff (t,y)  \sn^n_F$, $ \forall n \in \nn $. Since $ X $ 
	is $ \mc $, so it is globally Lipschitz. Let $ \beta > 0 $ be its Lipschitz constant we then have $\forall n \in \nn$
	\begin{equation*}
	\varphi_n (t) = \sn \int_0^t  \big ( X  (\ff (s,x)) - X (\ff (s,y)\big ) ds +x-y\sn^n_F \leqq \sn x-y \sn_F^n + \beta \int_{0}^{t} \varphi(s)ds. 
	\end{equation*}
	Thus, by Gronwall's inequality we obtain
	\begin{equation*}\label{les}
	\sn \ff (t,x) - \ff (t,y)  \sn^n_F \leqq e^{\beta \mid t \mid} \sn x -y \sn^n_F, \quad \forall n \in \nn.
	\end{equation*}
	Thereby $ \ff $ is Lipschitz continuous in the second variable and is jointly continuous.
	
	Now, define $ \mathbf{F}(t,x) \in \mathcal{L}_\rho(F)$
	to be the solution of the equations
	\begin{equation*}
	\dfrac{d \mathbf{F}(t,x)}{d t} = \dd X (\ff (t,x)) \circ \mathbf{F}(t,x), \quad \mathbf{F}(0,x) = \operatorname{id},
	\end{equation*}
	where $ \dd X (\ff (t,x)) : F \to F $ is derivative of $ X $ with respect to $ x $ at $ \ff (t,x) $.
	By Proposition~\ref{ht}  $ \mathbf{F}(t,x)$ exits and is well defined. Since the vector field $ \mathbf{F} \mapsto \dd X (\ff (t,x)) \circ \mathbf{F} $ on $ \mathcal{L}_\rho(F)$ is Lipschitz  in $ \mathbf{F} $, uniformly in $ (t,x) $ in a neighborhood of every $ (t_0,x_0) $, by the above
	argument it follows that  $ \mathbf{F}(t,x)$ is continuous in $ (t,x) $. We show that $ \dd \ff (t,x) = \mathbf{F}(t,x) $.
	Fix $ t \in I $, for $ h \in U $ define $ \psi (s,h) = \ff (s, x+h) - \ff (s,x)$, then
	\begin{align*}
	\psi (t,h)- \mathbf{F}(t,x)  (h) &= \int_{0}^{t} \big ( X (\ff (s,x+h)) - X (\ff (s,x))\big ) ds \\
	&- \int_{0}^{t} \big [\dd X (\ff (s,x)) \circ \mathbf{F}(s,x) \big ] (h)  ds \\
	&= \int_{0}^{t} \dd X (\ff (s,x)  ( \big[ (\psi (s,h) - \mathbf{F}(s,x) ( h)) \big ])  ds \\ &+ 
	\int_{0}^{t}  \big ( X (\ff (s,x+h)) - X (\ff (s,x) ) \\ &- \dd X (\ff (s,x))(\ff (s,x+h) - \ff (s,x)\big ) ds
	\end{align*}
	Since $ X $ is $ \mc $, for given $ \varepsilon>0 $ there is a $ \delta >0 $ such that $ \sn h \sn_F^n < \delta (\forall n \in \nn)  $ yields that the second term is less than 
	\begin{equation*}
	\int_{0}^{t}\varepsilon \sn \ff (s,x+h) - \ff (s,x) \sn_F^n, \quad \forall n \in \nn,
	\end{equation*}
	but by~\eqref{les} this integral is less than $B \varepsilon \sup_{n \in \nn}  \sn h \sn_F^n  $ for some positive constant 
	$ B $. Thus, by Gronwall's inequality we obtain $$ \sn \psi (t,h) - \mathbb{F}(t,h)(h) \sn_F^n \leqq \varepsilon C \sn h \sn_F^n, \quad  \forall n \in \nn,$$ where $ C $ is a positive constant. Whence $ \dd \ff (t,x)(h) = \mathbf{F}(t,x)(h)  $. Thus, both partial derivatives of $ \ff (t,x) $ exist and are continuous so $ \ff(t,x) $ is $ C^1 $. Moreover, $ \mathbf{F} $ is globally Lipschitz and $ x \mapsto \mathbf{F}(\cdot,x) $ is continuous therefore $ \ff(t,x) $ is $MC^1$. By induction on $ k $ we prove that $ \ff (t,x) $ is of class $ \mc $. By definition of $ \ff (t,x) $ 
	\begin{equation*}
	\dfrac{d }{d t} \ff (t,x)= X (\ff (t,x))
	\end{equation*}
	so
	\begin{equation*}
	\dfrac{d}{dt}\dfrac{d }{d t}\ff (t,x) = \dd X (\ff (t,x)) \big ( X (\ff (t,x)) \big)
	\end{equation*}
	and
	\begin{equation*}
	\dfrac{d}{dt} \dd \ff (t,x) = \dd X (\ff (t,x)) \big (\dd  \ff (t,x) \big).
	\end{equation*}
	The right-hand sides are $ MC^{k-1} $, so are the solutions by induction. Thus $ \ff (t,x) $ is $ \mc $.
\end{proof}
\section{Darboux charts}
In general for a Fr\'{e}chet manifold differential forms cannot be defined as the  sections of its cotangent bundle since
in general we can not define a manifold structure on the cotangent bundle, see \cite[Remark~I.3.9]{neeb}. To define differential forms we follow the approach of Neeb~\cite{neeb}.
\begin{defn}
	Let $M$ be a bounded Fr\'{e}chet manifold. A $p$-form $\omega$ on $M$ is a function 
	$\omega$ which associates to each $x \in M$ a $p$-linear 
	alternating map $\omega_x : T_x^p(M)\to \rr$ 
	such that in local coordinates  the
	map 
	$(x,v_1, \ldots, v_p) \mapsto \omega_{x}(v_1, \ldots, v_p)$
	is smooth. We write $\Omega^p(M,\rr)$ for the space of $p$-forms
	on $M$ and identify $\Omega^0(M,\rr)$ with the space $C^\infty(M,\rr)$ 
	of smooth functions. 
\end{defn}
The exterior differential 
$d_{dR} : \Omega^p(M,\rr) \to \Omega^{p+1}(M,\rr)$ is determined uniquely by the property that 
for each open subset $U \subset M$ we have for $X_0, \ldots, X_p \in \mathcal{ V}(U)$ 
in the space $C^\infty(U,\rr)$ the identity 
\begin{align*}
(d_{dR}\omega)(X_0, \ldots, X_p) 
&:= \sum_{i = 0}^p (-1)^i X_i.\omega(X_0, \ldots, \hat X_i, \ldots,X_p) \\
&+ \sum_{i < j} (-1)^{i+j}\omega([X_i, X_j], X_0, \ldots, \hat X_i,\ldots, \hat X_j, \ldots, X_p).
\end{align*}
Let  $ \omega \in \Omega^p(M,\rr)$,  $Y \in \mathcal{V}(M)$ and $ \ff_t $  the local flow of $ Y $. We define the usual {\it Lie derivative}  by 
$$
\mathcal{L}_Y \omega = \dfrac{d}{dt} (\ff_t^{*} \omega)  \mid_{t=0},
$$
which of course coincides by 
$$( \mathcal{L}_Y\omega)(X_1,\ldots, X_p) 
= Y.\omega(X_1, \ldots, X_p) - \sum_{j = 1}^p \omega(X_1, \ldots,
[Y, X_j], \ldots, X_p) $$
for $X_i \in \mathcal{V}(U)$, $U \subset M$ open. 
For each $X \in \mathcal{V}(M)$ and $p \geq 1$ we define a linear map 
$$ i_X  : \Omega^p(M,\rr) \to \Omega^{p-1}(M,\rr)  \quad \hbox{ with } \quad 
(i_X\omega)_x = i_{X(x)}\omega_x, $$
where 
$(i_v \omega_x)(v_1,\ldots, v_{p-1}) := \omega_x(v, v_1,\ldots, v_{p-1}).$
For $\omega \in \Omega^0(M,\rr) = C^\infty(M,\rr)$, we put $i_X \omega := 0$. 
For $X, Y \in \mathcal{V}(M)$, we have on $\Omega(M,\rr)$ the 
Cartan formulas (\cite[Proposition I.4.3]{neeb}) : 
\begin{equation*}
[\mathcal{L}_X, i_Y] = i_{[X,Y]}, \quad 
\mathcal{L}_X = d_{dR} \circ i_X + i_X \circ d_{dR} \quad \hbox{ and } \quad  
\mathcal{L}_X \circ d_{dR} = d_{dR} \circ \mathcal{L}_X.
\end{equation*} 
\begin{defn}
	Let $ M $ be a bounded Fr\'{e}chet manifold. We say that $ M $ is weakly symplectic if there exists a closed
	smooth 2-form $ \omega $ ($d_{dR} \omega = 0$) such that it is weakly non-degenerate i.e. for all $ x \in M $ and $ v_x \in T_x M $ 
	\begin{equation}\label{s}
	\omega_x(v_x,w_x)=0
	\end{equation}
	for all $ w_x \in T_xM $ implies $ v_x=0 $.
\end{defn}

The Darboux theorem is a local result so we consider the case
where $M$ is an open set $U$ of the Fr\'{e}chet model space $ F $. For the simplicity we suppose that  $0 \in U$.
Let $ x \in U $ be fixed and let $ F'_b $ be the strong dual of $ F $. Define the map $ \omega^{\hash}_x : F \to F'_b$ by 
$$ \langle w, \omega^{\hash}_x (v)  \rangle = \omega_x (w,v), $$
where $ \langle \cdot , \cdot \rangle $ is a duality pairing. Also, define $ H_x \coloneq \{ \omega_x (y,.) \mid y \in F \} $.
This is a subset of $ F^{'}_b $ and its topology is induced from it. 

\begin{lemma}
	Suppose $ x \in U $ is fixed and the model space $ F_x \simeq T_xM $ is nuclear. Then the map	$ \omega^{\hash}_x : F_x \to H_x $ is an isomorphism.
\end{lemma}
\begin{proof}
	Condition~\ref{s} implies that
	$  \omega_x^{\hash}$ is injective and by the definition of $ H_x $, it is surjective.
	The space $ F_x $ is nuclear so its strong dual is $ DFN $-space and barreled.
	The dual space is Mackey (cf. \cite[5.3.4]{kc}) and $ H_x $ inherits its topology (see \cite[0.4.2, 0.4.3]{kc} so is
	ultrabornological and barrelled (cf.\cite[8.6.9]{kc}). Thus, by the open mapping theorem~\cite[Theorem 4.35]{rr} the inverse mapping is continuous, so $\omega^{\hash}_x$ is isomorphism.

\end{proof}
We will need the following result.
\begin{lemma}\cite[(Poincar\'{e} Lemma) II.3.5]{neeb2}\label{po}
	Let $E$ be locally convex, $V$ a sequentially complete space and $U \subset E$
	an open subset which is star-shaped with respect to $0$. Let $\omega \in \Omega^{k+1}(U,V)$  be a $V$-valued closed
	$(k + 1)$-form. Then $\omega$ is exact. Moreover, $\omega = d_{dR}\alpha$ for some $ \alpha \in \Omega^{k}(U,V) $ with $ \alpha (0)=0 $ given by
	$$
	\alpha (x)(v_1,\cdots v_k) = \int_{0}^{1}t^k \omega (tk)(x,v_1, \cdots v_k)dt.
	$$
\end{lemma}
We assert the following theorem for some open neighborhood of $ 0 \in F $.
\begin{theorem}\label{dar}
	Let the Fr\'{e}chet model space $ F $ be nuclear.  Assume 
	\begin{enumerate}
		\item \label{1} there exits an open neighborhood $ \mathcal{U} $ of $ zero $ such that all the spaces $ H_x $ are locally identical and $ \omega_{x}^{t\hash} : F \to H$ is isomorphism for each $ t \in [0,1]$ and  $ x \in \mathcal{U} $.
		
		\item \label{2} for $ x \in \mathcal{U} $ the map $ (\omega_{x}^{t\hash})^{-1}: H \to F$ is a field of isomorphism of class $ MC^{\infty} $. 
	\end{enumerate} 
	Then $ \omega $ is locally isomorphic at zero to the constant form $ \omega(0) $. 
\end{theorem}
\begin{proof}
	On $ \mathcal{ U} $ define $ \omega^t = \omega_0 + t(\omega_0 - \omega)$ for $ t \in [0,1] $, where $ \omega_0= \omega(0) $.
	By Lemma~\ref{les} there exist 1-form $ \alpha $ locally such that $ d_{dR} \alpha = \omega_0 -\omega $ and $ \alpha(0) =0 $.
	Consider a time-dependent vector field $ X_t : \mathcal{U} \to F $ such that $$i_{X_t} \omega^t = -\alpha. $$
	Thus, $ \alpha = i_{X_1}\omega  $ so $ \alpha \in H $. By Condition~\ref{1} for
	$ x \in \mathcal{ U}$ and all $ t $, $\omega_{x}^{t\hash} $ is isomorphism hence 
	$$
	X_t  \coloneq (\omega_{x}^{t\hash})^{-1}\alpha
	$$
	is well defined. By Condition~\ref{2}, $ X_t $ is $ MC^{\infty} $ so Theorem~\ref{flow} implies that there exists an smooth isotopy $ \ff_t $
	generated by $ X_t $ and for $ t \in [0,1] $ it satisfies 
	\begin{equation}\label{fg}
	\ff_t^{*}\omega^t  = \omega_0	
	\end{equation}
	To solve~\ref{fg}, we need to solve 
	\begin{equation}
	\dfrac{d}{dt}\ff_t^{*}\omega^t = 0.
	\end{equation}
	We have by product rule of derivative and the Cartan formula 
	\begin{align}
	\dfrac{d}{dt}\ff_t^{*}\omega^t &= \ff_t^{*} (\mathcal{L}_{{X}_t} \omega^t)+ \ff_t^{*}\dfrac{d}{dt}\omega^t \\
	&= \ff_t^{*} ( \dfrac{d}{dt} \omega^t - d_{dR}(i_{X_t} \omega^t))\\
	&= \ff_t^* (-d \alpha + \omega_0 - \omega)= 0.
	\end{align}
	Thus, $ \ff_1^{*} \omega_1 =  \ff^{*}_0\omega_0$ so $ \ff_1^{*} \omega = \omega_0$.
\end{proof}
\begin{remk}
	In projective limit approach despite the fact that many interesting results can be recovered for Fr\'{e}chet manifolds
	there are some limitations. To construct geometric and topological objects we need to establish the
	existence of compatible projective limits of their Banach corresponding factors, this would not be easy in some cases and also we can not use
	some known results (e.g. the Poincare lemma for locally convex spaces) so it is imposed the additional condition in\cite[Theorem 4.2]{k3} for the existence of the required differential form as the Poincare lemma is not available in this setting. Also, we need a rather strong Lipschitz condition on mappings for the existence of local flows. In contrast, in metric approach we can apply known facts from the metric geometry and locally convex spaces that simplify proofs. There are some restrictions in this approach also; it is not easy to check $ \mc $-differentiability and the class of bounded maps can be very small. However as mentioned, manifolds of Riemannian metrics have the structure
	of nuclear bounded Fr\'{e}chet manifolds and Theorem~\ref{dar} can be used to study their cohomology, but it is not yet clear how to construct a symplectic structure that can be applied in this context.

\end{remk}

\end{document}